\documentclass{amsart}
\usepackage{amsmath}
\usepackage{amssymb}
\usepackage{amsfonts}
\allowdisplaybreaks[1]
\newtheorem{theorem}{Theorem}[section]
\theoremstyle{plain}

\newtheorem{corollary}{Corollary}[section]

\newtheorem{lemma}{Lemma}[section]

\newtheorem{remark}{Remark}

\numberwithin{equation}{section}

\begin{document}

 \title[A refinement of the companion of Ostrowski inequality]{A refinement of the companion of Ostrowski inequality for functions of bounded variation and applications}
\author[W. J. Liu]{Wenjun Liu}
\address[W. J. Liu]{College of Mathematics and Statistics\\
Nanjing University of Information Science and Technology \\
Nanjing 210044, China}
\email{wjliu@nuist.edu.cn}

\author[Y. Sun]{Yun Sun}
\address[Y. Sun]{College of Mathematics and Statistics\\
Nanjing University of Information Science and Technology \\
Nanjing 210044, China}


\begin{abstract}
In this paper we establish a refinement of the companion of Ostrowski inequality   for functions of bounded variation.
Applications for the trapezoid inequality, the mid-point inequality, and to probability density functions are also given.
\end{abstract}

\maketitle

\section{Introduction}
In 1938, Ostrowski  \cite{o1938} established the following interesting integral
inequality for differentiable mappings with bounded
derivatives:
\begin{theorem}\label{Th1.1}
Let $f:[a,b]\rightarrow\mathbb{R}$ be a differentiable mapping on
$(a,b)$ whose derivative is bounded on $(a,b)$ and denote
$\|f'\|_{\infty}=\displaystyle{\sup_{t\in(a,b)}}|f'(t)|<\infty$.
Then, for all $x\in[a,b]$,  one has the inequality
\begin{equation}\label{1.1}
\left|f(x)-\frac{1}{b-a}\int_{a}^{b}f(t)dt\right|\leq\left[\frac{1}{4}
+\frac{(x-\frac{a+b}{2})^{2}}{(b-a)^{2}}\right](b-a)\|f'\|_{\infty}.
\end{equation}
The constant $\frac{1}{4}$ is sharp in the sense that it can not be
replaced by a smaller one.
\end{theorem}

 This inequality has attracted considerable interest over the years, and many
authors proved generalizations,  modifications and applications of it.
In \cite{d2001mia}, Dragomir extended this result to the larger class of functions of
bounded variation,   as follows:

\begin{theorem}\label{Th1.1'}
Let $f:[a,b]\rightarrow\mathbb{R}$ be a function of bounded variation on
$[a,b]$. Denote by $\bigvee_a^b (f)$ its total variation on $[a, b]$. Then, for all $x\in[a,b]$, one has the inequality
\begin{equation}\label{1.2}
\left|f(x)-\frac{1}{b-a}\int_{a}^{b}f(t)dt\right|\leq\left[\frac{1}{2}
+\left|\frac{x-\frac{a+b}{2}}{b-a}\right|\right]\bigvee_a^b (f).
\end{equation}
The constant $\frac{1}{2}$ is sharp in the sense that it can not be
replaced by a smaller one.
\end{theorem}

The best inequality one can obtain from \eqref{1.1} and \eqref{1.2} is the midpoint inequality. The corresponding version for the generalised trapezoid inequality was obtained by Cerone, Dragomir, and Pearce in \cite{cdp2000}, from which one can derive the trapezoid inequality.
Recently, by using a critical Lemma, Dragomir \cite{d2008} proved
refinement of the generalised trapezoid and Ostrowski inequalities for functions of bounded variation, the particular cases of which provide  refinements of the trapezoid
and mid-point inequalities.

In \cite{gs2002}, Guessab and Schmeisser, in the effort of incorporating together
the mid-point and trapezoid inequalities, proved a companion of
Ostrowski's inequality.
Motivated by \cite{gs2002}, Dragomir \cite{d2002} proved some companions of Ostrowski's
inequality for functions of bounded variation, as follows:
\begin{theorem}\label{Th1.3}
Assume that the function $f:[a,b]\rightarrow\mathbb{R}$ is of bounded variation on
$[a,b]$.
Then the following inequalities
\begin{align}\label{1.3}
&\left|\frac{f(x)+f(a+b-x)}{2}-\frac{1}{b-a}\int_{a}^{b}f(t)dt\right|
\nonumber
\\
\leq&\frac{1}{b-a}\left[(x-a)\bigvee_a^x (f)+\left(\frac{a+b}{2}-x\right)\bigvee_x^{a+b-x} (f)+(x-a)\bigvee_{a+b-x}^b (f)\right]
\nonumber
\\
\leq&\left[\frac{1}{4}
+\left|\frac{x-\frac{3a+b}{4}}{b-a}\right|\right]\bigvee_a^b (f)
\end{align}
hold for all $x\in[a,\frac{a+b}{2}]$. The
constant $\frac{1}{4}$ is best possible.

The best inequality one can obtain from \eqref{1.3} is the trapezoid type inequality, namely,
\begin{align}
\left|\frac{f\left(\frac{3a+b}{4}\right)+f\left(\frac{a+3b}{4}\right)}{2}
-\frac{1}{b-a}\int_{a}^{b}f(t)dt\right|\leq \frac{1}{4} \bigvee_a^b (f).
\label{1.4}
\end{align}
Here the
constant $\frac{1}{4}$ is also the best possible.
\end{theorem}

For other related results, the reader may be refer to \cite{a20111, a20112, a20113, a1995, a2008, bdg2009, d2001, hn2011, l2008, l2010, lxw2010, l2007, l2009, s20101, s2010, ss2011, thd2008, thyc2011, u2003, v2011} and the
references therein.

The main aim of this paper is to establish a refinement of the companion of Ostrowski inequality \eqref{1.3} for functions of bounded variation.
Applications for the trapezoid inequality, the mid-point inequality, the trapezoid type inequality \eqref{1.4}, and to probability density functions are also given.

\section{Main results}

 To prove our main results, we need the following lemma, which is a slight improvement of \cite[Lemma 2.1]{d2008}.
\begin{lemma}\label{lem2.1}
 Let $u, f:[a,b]\rightarrow\mathbb{R}$. If $u$ is continuous on $[a,b]$ and $f$ is of bounded
variation on  $[c,b]\supseteq [a,b]$, then
\begin{align}
\left|\int_a^b u(t)df(t)\right|\le & \int_a^b |u(t)|d\left(\bigvee_c^t (f)\right)\nonumber\\
\le &\left[\bigvee_a^b (f)\right]^{\frac{1}{q}}\left[\int_a^b |u(t)|^pd\left(\bigvee_c^t (f)\right)\right]^{\frac{1}{p}}
\quad {\it if}\ p>1,\ \frac{1}{p}+\frac{1}{q}=1\nonumber\\
 \le & \max_{t\in [a,b]}|u(t)|\bigvee_a^b (f). \label{2.1}
\end{align}
\end{lemma}

\begin{proof}
See \cite[Lemma 2.1]{d2008}.
\end{proof}

The following result may be stated.

\begin{theorem}\label{Th2.1}
Assume that the function $f:[a,b]\rightarrow\mathbb{R}$ is of bounded variation on
$[a,b]$.
Then
\begin{align}\label{2.2}
\left|\frac{f(x)+f(a+b-x)}{2}-\frac{1}{b-a}\int_{a}^{b}f(t)dt\right|
\leq Q(x)
\end{align}
for all $x\in[a,\frac{a+b}{2}]$, where
\begin{align}\label{2.3}
Q(x):=&\frac{1}{b-a}\left[ 2 \left(\frac{3a+b}{4}-x\right)\bigvee_x^{a+b-x}(f)+\int_a^{\frac{a+b}{2}}\mathrm{sgn}(x-t)\bigvee_t^{a+b-t}(f)dt\right]\nonumber\\
\leq&\frac{1}{b-a}\left[(x-a)\bigvee_a^x (f)+\left(\frac{a+b}{2}-x\right)\bigvee_x^{a+b-x} (f)+(x-a)\bigvee_{a+b-x}^b (f)\right]
\nonumber
\\
\leq&\left[\frac{1}{4}
+\left|\frac{x-\frac{3a+b}{4}}{b-a}\right|\right]\bigvee_a^b (f).
\end{align}

We also have
\begin{align}\label{2.4}
Q(x)\leq &\frac{1}{b-a}\left[\bigvee_a^b (f)\right]^{\frac{1}{q}}\left\{   \left[  \left(\frac{a+b}{2}-x\right)^p- \left(x-a\right)^p   \right]\bigvee_x^{a+b-x}(f)\right.\nonumber\\&\left.+p\int_a^{\frac{a+b}{2}}r_p(x,t)\mathrm{sgn}(x-t)\bigvee_t^{a+b-t}(f)dt\right\}^\frac{1}{p}\nonumber\\
\leq&\frac{1}{b-a}\left[\bigvee_a^b (f)\right]^{\frac{1}{q}}\left\{   (x-a)^p\bigvee_a^x (f)+\left(\frac{a+b}{2}-x\right)^p\bigvee_x^{a+b-x} (f)+(x-a)^p\bigvee_{a+b-x}^b (f)\right\}^\frac{1}{p}\nonumber\\
\leq&\left[\frac{1}{4}
+\left|\frac{x-\frac{3a+b}{4}}{b-a}\right|\right]\bigvee_a^b (f),
\end{align}
where $p>1,\ \frac{1}{p}+\frac{1}{q}=1$ and $r_p:[a,b]^2\rightarrow\mathbb{R}$ with
\begin{align}
r_p(x,t):=\left\{ {\begin{aligned} & (t-a)^{p-1}, && t\in[a,x], \\
& \left(\frac{a+b}{2}-x\right)^{p-1}, && t\in\left(x, \frac{a+b}{2}\right].
\end{aligned}}\right. \label{2.5}
\end{align}
\end{theorem}
\begin{proof}
Define the kernel $K(x,t)$ by
\begin{align}
K(x,t):=\left\{ {\begin{aligned} & t-a, && t\in[a,x], \\
& t-\frac{a+b}{2}, && t\in(x,a+b-x], \\
& t-b, && t\in(a+b-x,b],
\end{aligned}}\right. \label{2.6}
\end{align}
for all $x\in[a,\frac{a+b}{2}]$. Using the integration by parts formula for Riemann-Stieltjes integrals, we obtain
(see  \cite{d2002})
\begin{align}
\frac{1}{b-a}\int_{a}^{b}K(x,t)df(t)=\frac{f(x)+f(a+b-x)}{2}-\frac{1}{b-a}\int_{a}^{b}f(t)dt.
\label{2.7}
\end{align}

 Now, if $f$ is of bounded variation on
$[a,b]$,  then on taking the modulus, applying the first inequality
in \eqref{2.1} and making a substitution of the form  $t=a+b-s$, we deduce
\begin{align*}
&\left|\int_{a}^{b}K(x,t)df(t)\right|\\
\le &
\left|\int_a^x (t-a)df(t)\right|+\left|\int_x^{a+b-x} \left(t-\frac{a+b}{2}\right)df(t)\right|+\left|\int_{a+b-x}^b (t-b)df(t)\right|\\
\le &
 \int_a^x |t-a|d\left(\bigvee_a^t (f)\right)+\int_x^{a+b-x} \left|t-\frac{a+b}{2}\right|d\left(\bigvee_a^t (f)\right)+\int_{a+b-x}^b |t-b|d\left(\bigvee_a^t (f)\right)\\
=& \int_a^x (t-a)d\left(\bigvee_a^t (f)\right)-\int_x^\frac{a+b}{2} \left(t-\frac{a+b}{2}\right)d\left(\bigvee_a^t (f)\right)\\
&+\int_\frac{a+b}{2}^{a+b-x} \left(t-\frac{a+b}{2}\right)d\left(\bigvee_a^t (f)\right)+\int_{a+b-x}^b (b-t)d\left(\bigvee_a^t (f)\right)\\
=& (x-a)\bigvee_a^x (f)-\int_a^x \left(\bigvee_a^t (f)\right)dt+\left(x-\frac{a+b}{2}\right)\bigvee_a^x (f)+\int_x^\frac{a+b}{2} \left(\bigvee_a^t (f)\right)dt\\
&+\left(\frac{a+b}{2}-x\right)\bigvee_a^{a+b-x}(f)-\int_\frac{a+b}{2}^{a+b-x} \left(\bigvee_a^t (f)\right)dt-(x-a)\bigvee_a^{a+b-x}(f)+\int_{a+b-x}^b \left(\bigvee_a^t (f)\right)dt\\
=&  2 \left(\frac{3a+b}{4}-x\right)\bigvee_x^{a+b-x}(f)-\int_a^{\frac{a+b}{2}}\mathrm{sgn}(x-t)\bigvee_a^t(f)dt
-\int_{\frac{a+b}{2}}^b\mathrm{sgn}(a+b-x-t)\bigvee_a^t(f)dt\\
=&  2 \left(\frac{3a+b}{4}-x\right)\bigvee_x^{a+b-x}(f)-\int_a^{\frac{a+b}{2}}\mathrm{sgn}(x-t)\bigvee_a^t(f)dt
+\int_a^{\frac{a+b}{2}} \mathrm{sgn}(x-s)\bigvee_a^{a+b-s}(f)ds\\
=&  2 \left(\frac{3a+b}{4}-x\right)\bigvee_x^{a+b-x}(f)+\int_a^{\frac{a+b}{2}}\mathrm{sgn}(x-t)\bigvee_t^{a+b-t}(f)dt
\\
=& (b-a)Q(x),
\end{align*}
for all $x\in[a,\frac{a+b}{2}]$, and the inequality \eqref{2.2} is proved.

Now, since $\bigvee_a^\cdot$ is monotonic nondecreasing on $[a, b]$, then
$$\int_a^x \left(\bigvee_a^t (f)\right)dt\ge 0,\quad \int_x^\frac{a+b}{2} \left(\bigvee_a^t (f)\right)dt\le \left(\frac{a+b}{2}-x\right)\bigvee_a^\frac{a+b}{2} (f),$$
$$\int_\frac{a+b}{2}^{a+b-x} \left(\bigvee_a^t (f)\right)dt\ge \left(\frac{a+b}{2}-x\right)\bigvee_a^\frac{a+b}{2} (f)\quad {\it and}\quad
\int_{a+b-x}^b \left(\bigvee_a^t (f)\right)dt\le (x-a)\bigvee_{a}^b (f),$$
for  all $x\in[a,\frac{a+b}{2}]$, which gives
\begin{align*}
 (b-a)Q(x)\le &  2 \left(\frac{3a+b}{4}-x\right)\bigvee_x^{a+b-x}(f)+(x-a)\bigvee_a^{b}(f) \\
 =& (x-a)\bigvee_a^x (f)+\left(\frac{a+b}{2}-x\right)\bigvee_x^{a+b-x} (f)+(x-a)\bigvee_{a+b-x}^b (f)
\end{align*}
and the inequality \eqref{2.3} is proved.

Utilising the second part of the second inequality in \eqref{2.1} and  H\"older's inequality,  we deduce that
\begin{align}\label{2.8}
 &(b-a)Q(x)\nonumber\\\le &
 \left[\bigvee_a^x (f)\right]^{\frac{1}{q}}\left[\int_a^x |t-a|^pd\left(\bigvee_a^t (f)\right)\right]^{\frac{1}{p}}
 +\left[\bigvee_x^{a+b-x} (f)\right]^{\frac{1}{q}}\left[\int_x^{a+b-x} \left|t-\frac{a+b}{2}\right|^pd\left(\bigvee_a^t (f)\right)\right]^{\frac{1}{p}}\nonumber\\&+\left[\bigvee_{a+b-x}^b (f)\right]^{\frac{1}{q}}\left[\int_{a+b-x}^b |t-b|^pd\left(\bigvee_a^t (f)\right)\right]^{\frac{1}{p}}\nonumber\\
 \le & \left[\bigvee_a^b (f)\right]^{\frac{1}{q}}\left[\int_a^x |t-a|^pd\left(\bigvee_a^t (f)\right)+\int_x^{a+b-x} \left|t-\frac{a+b}{2}\right|^pd\left(\bigvee_a^t (f)\right)\right.\nonumber\\&\left.+\int_{a+b-x}^b |t-b|^pd\left(\bigvee_a^t (f)\right)
 \right]^{\frac{1}{p}}.
\end{align}
Now, observe that
\begin{align*}
 R(x):= &  \int_a^x |t-a|^pd\left(\bigvee_a^t (f)\right)+\int_x^{a+b-x} \left|t-\frac{a+b}{2}\right|^pd\left(\bigvee_a^t (f)\right)+\int_{a+b-x}^b |t-b|^pd\left(\bigvee_a^t (f)\right) \\
 =& \int_a^x (t-a)^pd\left(\bigvee_a^t (f)\right)+\int_x^\frac{a+b}{2} \left(\frac{a+b}{2}-t\right)^pd\left(\bigvee_a^t (f)\right)\\
&+\int_\frac{a+b}{2}^{a+b-x} \left(t-\frac{a+b}{2}\right)^pd\left(\bigvee_a^t (f)\right)+\int_{a+b-x}^b (b-t)^pd\left(\bigvee_a^t (f)\right)\\
=& (x-a)^p\bigvee_a^x (f)-p\int_a^x (t-a)^{p-1} \bigvee_a^t (f) dt-\left(\frac{a+b}{2}-x\right)^p\bigvee_a^x (f)\\
&+p\int_x^\frac{a+b}{2} \left(\frac{a+b}{2}-t\right)^{p-1} \bigvee_a^t (f) dt+\left(\frac{a+b}{2}-x\right)^p\bigvee_a^{a+b-x}(f)\\
&-p\int_\frac{a+b}{2}^{a+b-x}\left(t-\frac{a+b}{2}\right)^{p-1}  \bigvee_a^t (f) dt-(x-a)^p\bigvee_a^{a+b-x}(f)+p\int_{a+b-x}^b(b-t)^{p-1}  \bigvee_a^t (f) dt\\
=&  \left[  \left(\frac{a+b}{2}-x\right)^p- \left(x-a\right)^p   \right]\bigvee_x^{a+b-x}(f)\\
&-p\int_a^x(t-a)^{p-1}\bigvee_a^t(f)dt
+p\int_x^{\frac{a+b}{2}} \left(\frac{a+b}{2}-t\right)^{p-1}\bigvee_a^t(f)dt\\
&-p\int_x^{\frac{a+b}{2}} \left(\frac{a+b}{2}-s\right)^{p-1}\bigvee_a^{a+b-s}(f)ds+p\int_a^x(s-a)^{p-1}\bigvee_a^{a+b-s}(f)ds\\
=&  \left[  \left(\frac{a+b}{2}-x\right)^p- \left(x-a\right)^p   \right]\bigvee_x^{a+b-x}(f)-p\int_a^{\frac{a+b}{2}}r_p(x,t)\mathrm{sgn}(x-t)\bigvee_a^t(f)dt\\
&
+p\int_a^{\frac{a+b}{2}}r_p(x,s)\mathrm{sgn}(x-s)\bigvee_a^{a+b-s}(f)ds\\
=&  \left[  \left(\frac{a+b}{2}-x\right)^p- \left(x-a\right)^p   \right]\bigvee_x^{a+b-x}(f)  +p\int_a^{\frac{a+b}{2}}r_p(x,t)\mathrm{sgn}(x-t)\bigvee_t^{a+b-t}(f)dt,
\end{align*}
where $r_p$ is given in \eqref{2.5}. Utilising \eqref{2.8}, we deduce the first part of \eqref{2.4}.

Since $\bigvee_a^\cdot$ is monotonic nondecreasing on $[a, b]$, we have
\begin{align*}
R(x)\le &  \left[  \left(\frac{a+b}{2}-x\right)^p- \left(x-a\right)^p   \right]\bigvee_x^{a+b-x}(f)
+p\left[\int_x^{\frac{a+b}{2}}\left(\frac{a+b}{2}-t\right)^{p-1} dt \right]\bigvee_a^{\frac{a+b}{2}}(f)  \\
&-p\left[\int_x^{\frac{a+b}{2}}\left(\frac{a+b}{2}-s\right)^{p-1} ds \right]\bigvee_a^{\frac{a+b}{2}}(f)
+p\left[\int_a^x\left(s-a\right)^{p-1} ds \right]\bigvee_a^b(f) \\
 =& \left[  \left(\frac{a+b}{2}-x\right)^p- \left(x-a\right)^p   \right]\bigvee_x^{a+b-x}(f) +(x-a)^p\bigvee_a^{b}(f)\\
 =&  (x-a)^p\bigvee_a^x (f)+\left(\frac{a+b}{2}-x\right)^p\bigvee_x^{a+b-x} (f)+(x-a)^p\bigvee_{a+b-x}^b (f),
\end{align*}
which proves \eqref{2.4}.
\end{proof}

\begin{corollary}
Under the assumptions of Theorem \ref{Th2.1} with $x=\frac{3a+b}{4}$, we have the refined  
trapezoid type inequalities
\begin{align}
&\left|\frac{f\left(\frac{3a+b}{4}\right)+f\left(\frac{a+3b}{4}\right)}{2}
-\frac{1}{b-a}\int_{a}^{b}f(t)dt\right|\nonumber\\
\leq & \frac{1}{b-a}  \int_a^{\frac{a+b}{2}}\mathrm{sgn}\left(\frac{3a+b}{4}-t\right)\bigvee_t^{a+b-t}(f)dt \nonumber\\
\leq & \frac{1}{b-a} \left[\bigvee_a^b (f)\right]^{\frac{1}{q}}\left[  p\int_a^{\frac{a+b}{2}}r_p(x,t)\mathrm{sgn}\left(\frac{3a+b}{4}-t\right)\bigvee_t^{a+b-t}(f)dt\right]^\frac{1}{p}\nonumber\\
\leq & \frac{1}{4} \bigvee_a^b (f),
\label{2.9}
\end{align}
where $r_p$ is given in \eqref{2.5}.
\end{corollary}

\begin{corollary}
Under the assumptions of Theorem \ref{Th2.1} with $x=a$, we have the refined  
trapezoid inequalities
\begin{align}
&\left|\frac{f\left(a\right)+f\left(b\right)}{2}
-\frac{1}{b-a}\int_{a}^{b}f(t)dt\right|\nonumber\\
\leq &\frac{1}{2} \bigvee_a^b (f)- \frac{1}{b-a}  \int_a^{\frac{a+b}{2}} \left(\bigvee_t^{a+b-t}(f)\right)dt \nonumber\\
\leq & \frac{1}{b-a} \left[\bigvee_a^b (f)\right]^{\frac{1}{q}}\left[  \left(\frac{b-a}{2}\right)^p \bigvee_a^b (f) -p\int_a^{\frac{a+b}{2}} \left(\frac{a+b}{2}-t\right)^{p-1}\bigvee_t^{a+b-t}(f)dt\right]^\frac{1}{p}\nonumber\\
\leq & \frac{1}{2} \bigvee_a^b (f).
\label{2.10}
\end{align}
\end{corollary}

\begin{corollary}
Under the assumptions of Theorem \ref{Th2.1} with $x=\frac{a+b}{2}$, we have the refined  
midpoint inequalities
\begin{align}
&\left|f\left(\frac{a+b}{2}\right)
-\frac{1}{b-a}\int_{a}^{b}f(t)dt\right|\nonumber\\
\leq &\frac{1}{b-a}  \int_a^{\frac{a+b}{2}} \left(\bigvee_t^{a+b-t}(f)\right)dt \nonumber\\
\leq & \frac{1}{b-a} \left[\bigvee_a^b (f)\right]^{\frac{1}{q}}\left[ p\int_a^{\frac{a+b}{2}} \left(t-a\right)^{p-1}\bigvee_t^{a+b-t}(f)dt\right]^\frac{1}{p}\nonumber\\
\leq & \frac{1}{2} \bigvee_a^b (f).
\label{2.11}
\end{align}
\end{corollary}

\begin{remark} \label{re1}
The inequalities \eqref{2.2}-\eqref{2.4} provide a refinement of the companion of Ostrowski inequality \eqref{1.3}, and the inequalities \eqref{2.10}, \eqref{2.11} and \eqref{2.9} are refinements of the trapezoid inequality, the mid-point inequality  and  the trapezoid type inequality \eqref{1.4}, respectively, which were obtained in \cite{d1999}, \cite{d2000} and  \cite{d2002},   respectively.
\end{remark}

A new inequality of Ostrowski's type may be stated as follows:
\begin{corollary}\label{co2.4}
Let $f$ be as in Theorem \ref{Th2.1}. Additionally, if $f$ is symmetric about the line $x=\frac{a+b}{2}$,
i.e., $f (a + b - x) = f (x)$,
 then for all $x\in[a,\frac{a+b}{2}]$ we have
\begin{equation}
\left|f(x)-\frac{1}{b-a}\int_{a}^{b}f(t)dt\right|\leq
Q(x),
\end{equation}
where $Q(x)$ satisfies \eqref{2.3} and \eqref{2.4}.
\end{corollary}

\begin{remark} \label{re2}
Under the assumptions of Corollary \ref{co2.4} with $x=a$, we have
\begin{align}
&\left|f(a)
-\frac{1}{b-a}\int_{a}^{b}f(t)dt\right|\nonumber\\
\leq &\frac{1}{2} \bigvee_a^b (f)- \frac{1}{b-a}  \int_a^{\frac{a+b}{2}} \left(\bigvee_t^{a+b-t}(f)\right)dt \nonumber\\
\leq & \frac{1}{b-a} \left[\bigvee_a^b (f)\right]^{\frac{1}{q}}\left[  \left(\frac{b-a}{2}\right)^p \bigvee_a^b (f) -p\int_a^{\frac{a+b}{2}} \left(\frac{a+b}{2}-t\right)^{p-1}\bigvee_t^{a+b-t}(f)dt\right]^\frac{1}{p}\nonumber\\
\leq & \frac{1}{2} \bigvee_a^b (f).
\label{2.13}
\end{align}
\end{remark}

\section{Application to probability density functions}

Now, let $X$ be a random variable taking values in the finite interval
$[a,b]$, with the probability density function $f : [a, b]\rightarrow [0, 1]$ and with
the cumulative distribution function $$F (x) = Pr (X \leq  x) = \int_a^x f (t) dt.$$

The following results hold:
\begin{theorem}\label{Th4.1}
With the assumptions of Theorem \ref{Th2.1}, we have
\begin{equation}
\left|\frac{1}{2}[F(x)+F(a+b-x)]-\frac{b-E(X)}{b-a}\right|\leq T(x)\label{4.1}
\end{equation}
for all $x\in[a,\frac{a+b}{2}]$, where $E (X)$ is the expectation of $X$ and
\begin{align}\label{4.2}
T(x):=&\frac{1}{b-a}\left[ 2 \left(\frac{3a+b}{4}-x\right)[F(a+b-x)-F(x)]+\int_a^{\frac{a+b}{2}}\mathrm{sgn}(x-t)[F(a+b-t)-F(t)]dt\right]\nonumber\\
\leq&\frac{1}{b-a}\left[2 \left(\frac{3a+b}{4}-x\right)[F(a+b-x)-F(x)]+(x-a)\right]
\nonumber
\\
\leq&\frac{1}{4}
+\left|\frac{x-\frac{3a+b}{4}}{b-a}\right|.
\end{align}

We also have
\begin{align}\label{4.3}
T(x)\leq &\frac{1}{b-a} \left\{   \left[  \left(\frac{a+b}{2}-x\right)^p- \left(x-a\right)^p   \right][F(a+b-x)-F(x)]\right.\nonumber\\&\left.+p\int_a^{\frac{a+b}{2}}r_p(x,t)\mathrm{sgn}(x-t)[F(a+b-t)-F(t)]dt\right\}^\frac{1}{p}\nonumber\\
\leq&\frac{1}{b-a} \left\{    \left[  \left(\frac{a+b}{2}-x\right)^p- \left(x-a\right)^p   \right][F(a+b-x)-F(x)]+(x-a)^p \right\}^\frac{1}{p}\nonumber\\
\leq& \frac{1}{4}
+\left|\frac{x-\frac{3a+b}{4}}{b-a}\right|,
\end{align}
where $p>1,\ \frac{1}{p}+\frac{1}{q}=1$ and  $r_p$ is given in \eqref{2.5}.
\end{theorem}
\begin{proof}  By \eqref{2.2}-\eqref{2.4} on choosing $f = F$ and taking into account
$$E(X)=\int_a^b t dF(t)=b-\int_a^b F(t)dt,$$
we obtain \eqref{4.1}-\eqref{4.3}.
\end{proof}

\begin{corollary}
Under the  assumptions of Theorem \ref{Th4.1} with $x=\frac{3a+b}{4}$, we have
\begin{align}\label{4.4}
&\left|\frac{1}{2}\left[F\left(\frac{3a+b}{4}\right)+F\left(\frac{a+3b}{4}\right)\right]-\frac{b-E(X)}{b-a}\right|\nonumber\\
\leq & \frac{1}{b-a}  \int_a^{\frac{a+b}{2}}\mathrm{sgn}\left(\frac{3a+b}{4}-t\right)[F(a+b-t)-F(t)]dt \nonumber\\
\leq & \frac{1}{b-a} \left[  p\int_a^{\frac{a+b}{2}}r_p(x,t)\mathrm{sgn}\left(\frac{3a+b}{4}-t\right)[F(a+b-t)-F(t)]dt\right]^\frac{1}{p}\nonumber\\
\leq & \frac{1}{4},
\end{align}
where $r_p$ is given in \eqref{2.5}.
\end{corollary}

\begin{remark} \label{re3}
The inequalities \eqref{4.1}-\eqref{4.4} provide refinements of the inequalities  given in \cite[Theorem 5]{d2002} and \cite[Corollary 4]{d2002}.
\end{remark}

\subsection*{Acknowledgments}
This work was partly supported by the National Natural Science Foundation
of China (Grant No. 40975002) and the Natural Science Foundation of the Jiangsu
Higher Education Institutions (Grant No. 09KJB110005).


\begin{thebibliography}{00}


\bibitem{a20111}
M. W. Alomari, A companion of Ostrowski's inequality with applications, Transylv. J. Math. Mech. {\bf 3} (2011), no.~1, 9--14.



\bibitem{a20112}
M. W. Alomari,
A companion of Ostrowski's inequality  for mappings whose first derivatives are bounded and applications in numerical integration, RGMIA Res. Rep. Coll., 14 (2011) article 57.

\bibitem{a20113}
M. W. Alomari,
A generalization of companion inequality of  Ostrowski's type for mappings whose first derivatives are bounded and applications in numerical integration, RGMIA Res. Rep. Coll., 14 (2011) article 58.

\bibitem{a1995}
G. A. Anastassiou, Ostrowski type inequalities, Proc. Amer. Math. Soc. {\bf 123} (1995), no.~12, 3775--3781. 

\bibitem{a2008}
G. A. Anastassiou\ and\ J. A. Goldstein, Higher order Ostrowski type inequalities over Euclidean domains, J. Math. Anal. Appl. {\bf 337} (2008), no.~2, 962--968.


\bibitem{bdg2009}
N. S. Barnett, S. S. Dragomir\ and\ I. Gomm, A companion for the Ostrowski and the generalised trapezoid inequalities, Math. Comput. Modelling {\bf 50} (2009), no.~1-2, 179--187.

\bibitem{cdp2000}
P. Cerone, S. S. Dragomir\ and\ C. E. M. Pearce, A generalized trapezoid inequality for functions of bounded variation, Turkish J. Math. {\bf 24} (2000), no.~2, 147--163.

\bibitem{d1999}
S. S. Dragomir, The Ostrowski integral inequality for mappings of bounded variation, Bull. Austral. Math. Soc. {\bf 60} (1999), no.~3, 495--508.

\bibitem{d2000}
S. S. Dragomir, On the midpoint quadrature formula for mappings with bounded variation and applications, Kragujevac J. Math. {\bf 22} (2000), 13--19.

\bibitem{d2001mia}
S. S. Dragomir, On the Ostrowski's integral inequality for mappings with bounded variation and applications, Math. Inequal. Appl. {\bf 4} (2001), no.~1, 59--66.

\bibitem{d2002}
S. S. Dragomir, A companions of Ostrowski's inequality for functions of bounded variation and applications, RGMIA Res. Rep. Coll., 5 (2002), Supp., article 28.



\bibitem{d2008}
S. S. Dragomir, Refinements of the generalised trapezoid and Ostrowski inequalities for functions of bounded variation, Arch. Math. (Basel) {\bf 91} (2008), no.~5, 450--460.





\bibitem{d2001}
J. Duoandikoetxea, A unified approach to several inequalities involving functions and derivatives, Czechoslovak Math. J. {\bf 51(126)} (2001), no.~2, 363--376.

\bibitem{gs2002}
A. Guessab\ and\ G. Schmeisser, Sharp integral inequalities of the Hermite-Hadamard type, J. Approx. Theory {\bf 115} (2002), no.~2, 260--288.

\bibitem{hn2011}
V. N. Huy\ and\ Q. -A. Ng\^o, New bounds for the Ostrowski-like type inequalities, Bull. Korean Math. Soc. {\bf 48} (2011), no.~1, 95--104.

\bibitem{l2008}
W. J. Liu,  Several error inequalities for a quadrature formula with a parameter and applications, Comput. Math. Appl. {\bf 56} (2008), no.~7, 1766--1772.

\bibitem{l2010}
W. J. Liu, Some weighted integral inequalities with a parameter and applications, Acta Appl. Math. {\bf 109} (2010), no.~2, 389--400.


\bibitem{lxw2010}
W. J. Liu, Q. L. Xue\ and\ S. F. Wang, New generalization of perturbed Ostrowski type inequalities and applications, J. Appl. Math. Comput. {\bf 32} (2010), no.~1, 157--169.

\bibitem{l2007}
Z. Liu, Note on a paper by N. Ujevi\'c, Appl. Math. Lett. {\bf 20} (2007), no.~6, 659--663.

\bibitem{l2009}
Z. Liu, Some companions of an Ostrowski type inequality and applications, JIPAM. J. Inequal. Pure Appl. Math. {\bf 10} (2009), no.~2, Article 52, 12 pp.



\bibitem{o1938}
A. Ostrowski, \"Uber die Absolutabweichung einer differentiierbaren Funktion von ihrem Integralmittelwert, Comment. Math. Helv. {\bf 10} (1937), no.~1, 226--227.

\bibitem{s20101}
M. Z. Sarikaya, On the Ostrowski type integral inequality, Acta Math. Univ. Comenian. (N.S.) {\bf 79} (2010), no.~1, 129--134.

\bibitem{s2010}
M. Z. Sarikaya, New weighted Ostrowski and \v Ceby\v sev type inequalities on time scales, Comput. Math. Appl. {\bf 60} (2010), no.~5, 1510--1514.

\bibitem{ss2011}
E. Set\ and\ M. Z. Sar\i kaya, On the generalization of Ostrowski and Gr\"uss type discrete inequalities, Comput. Math. Appl. {\bf 62} (2011), no.~1, 455--461.

\bibitem{thd2008}
K.-L. Tseng, S.-R. Hwang\ and\ S. S. Dragomir, Generalizations of weighted Ostrowski type inequalities for mappings of bounded variation and their applications, Comput. Math. Appl. {\bf 55} (2008), no.~8, 1785--1793.

\bibitem{thyc2011}
K.-L. Tseng,  S.-R. Hwang, G.-S. Yang\ and\ Y.-M. Chou, Weighted Ostrowski integral inequality for mappings of bounded variation, Taiwanese J. Math. {\bf 15} (2011), no.~2, 573--585.

\bibitem{u2003}
N. Ujevi\'c, New bounds for the first inequality of Ostrowski-Gr\"uss type and applications, Comput. Math. Appl. {\bf 46} (2003), no.~2-3, 421--427.

\bibitem{v2011}
S. W. Vong, A note on some Ostrowski-like type inequalities, Comput. Math. Appl. {\bf 62} (2011), no.~1, 532--535.



















\end{thebibliography}
\end{document}